\documentclass[12pt, a4paper, reqno]{article}
\usepackage{verbatim}

\usepackage{slashbox}

\usepackage{amsmath, amscd, amsthm, amscd, amsfonts, amssymb, graphicx, color, soul, enumerate, xcolor,  mathrsfs, latexsym}
\usepackage{pstricks,pst-node,pst-tree}
\usepackage[bookmarksnumbered, colorlinks, plainpages,backref]{hyperref}
\hypersetup{colorlinks=true,linkcolor=darkblue, anchorcolor=green, citecolor=midblue, urlcolor=darkblue, filecolor=magenta, pdftoolbar=true}

\usepackage[all]{xy}

 \definecolor{cupgreen}{rgb}{0,0.498,0.208}
  \definecolor{cupblue}{rgb}{0,0,.5}
  \definecolor{cupred}{rgb}{1,0.04,0}
  \definecolor{cuppink}{rgb}{0.925,0,0.545}
  \definecolor{cupmagenta}{rgb}{0.624,0.161,0.424}
  \definecolor{cupbrown}{rgb}{0.71,0.212,0.133}

  \definecolor{cupgreen}{rgb}{0,0,0}
  \definecolor{cupblue}{rgb}{0,0,0}
  \definecolor{cupred}{rgb}{0,0,0}
  \definecolor{cuppink}{rgb}{0,0,0}
  \definecolor{cupmagenta}{rgb}{0,0,0}
  \definecolor{cupbrown}{rgb}{0,0,0}

\definecolor{TITLE}{rgb}{0,0,0}
\definecolor{midblue}{rgb}{0.00,0.0,0.80}
\definecolor{darkblue}{rgb}{0.00,0.00,0.45}
\definecolor{SECTION}{rgb}{0.50,0.00,1.00}
\definecolor{THM}{rgb}{0.8,0,0.1}
\definecolor{SEC}{rgb}{0,0,1}

\newcommand{\aut}{\mathrm{Aut}}

\makeatother
\normalsize
\newtheorem{theorem}{{\color{THM} Theorem}}[section]

\DeclareRobustCommand{\stirling}{\genfrac\{\}{0pt}{}}

\newtheorem{lemma}[theorem]{{\color{THM}Lemma}}
\newtheorem{proposition}[theorem]{{\color{THM}Proposition}}
\newtheorem{corollary}[theorem]{{\color{THM}Corollary}}

\theoremstyle{definition}

\numberwithin{equation}{section}

\title{Number of Distinguishing Colorings and Partitions}

\author{Bahman Ahmadi, Fatemeh Alinaghipour\\
\small and \\
Mohammad Hadi Shekarriz\footnote{Corresponding author}\\
\small Department of Mathematics, Shiraz University, 71454, Shiraz, Iran.\\
\small\tt bahman.ahmadi, fatemeh.naghipour, mshekarriz@shirazu.ac.ir\\
\date {}}

\begin{document}

\maketitle

\begin{abstract}
A vertex coloring of a graph $G$ is called distinguishing (or symmetry breaking) if no non-identity automorphism of $G$ preserves it, and the distinguishing number, shown by $D(G)$, is the smallest number of colors required for such a coloring. This paper is about counting non-equivalent distinguishing colorings of graphs with $k$ colors. A parameter, namely $\Phi_k (G)$, which is the number of non-equivalent distinguishing colorings of a graph $G$ with at most $k$ colors, is shown here to have an application in calculating the distinguishing number of the lexicographic product and the $X$-join of graphs. We study this index (and some other similar indices) which is generally difficult to calculate. Then, we show that if one knows the distinguishing threshold of a graph $G$, which is the smallest number of colors $\theta(G)$ so that, for $k\geq \theta(G)$, every $k$-coloring of $G$ is distinguishing, then, in some special cases, counting the number of distinguishing colorings with $k$ colors is very easy. We calculate $\theta(G)$ for some classes of graphs including the Kneser graph $K(n,2)$. We then turn to vertex partitioning by studying the distinguishing coloring partition of a graph $G$; a partition of vertices of $G$ which induces a distinguishing coloring for $G$. There, we introduce $\Psi_k (G)$ as the number of non-equivalent distinguishing coloring partitions with at most $k$ cells, which is a generalization to its distinguishing coloring counterpart. 
\end{abstract}

\textbf{Keywords:} {distinguishing coloring, distinguishing threshold, distinguishing partition, distinguishing coloring partition}

\bigskip

\textbf{Mathematics Subject Classification}: 05C15, 05C25, 05C30

\section{Introduction}

Breaking symmetries of graphs via vertex coloring is a subject initiated by Babai's work \cite{ba-1977} in 1977. There, he introduced an \emph{asymmetric coloring of a graph}, and proved that a tree has an asymmetric coloring with two colors if all vertices have the same degree $\alpha \geq 2$, where $\alpha$ can be an arbitrary finite or infinite cardinal. The concept was later called \emph{distinguishing coloring} in the literature, since the appearance of \cite{ac} by Albertson and Collins in 1996.

This paper is about counting \emph{non-equivalent} distinguishing colorings and partitions of a given graph with $k$ colors. Considering this number in case of 2 colors is as old as symmetry breaking in graphs; in 1977 Babai \cite{ba-1977} tried to count distinguishing colorings of infinite trees, while in 1991 Polat and Sabidussi \cite{Polat} tried to count \emph{essentially different asymmetrising sets} in finite and infinite trees which are \emph{distinguishing (coloring) partitions with $2$ cells} in our terminology (see Section \ref{dist-part}).

A vertex coloring of a graph $G$ is called \emph{distinguishing} if it is only preserved by the identity automorphism; in this case, we say that the coloring \emph{breaks} all the symmetries of $G$. By a $k$-\textit{distinguishing coloring} we mean a distinguishing coloring which uses exactly $k$ colors. The {\it distinguishing number} of a graph $G$, denoted $D(G)$, is the smallest number $d$ such that there exists a distinguishing vertex coloring of $G$ with $d$ colors. A graph $G$ is called {\it $d$-distinguishable} if there exists a distinguishing vertex coloring with $d$ colors \cite{ac}. The distinguishing number of some important classes of graphs are as follows: $D(K_{n})=n$, $D(K_{n,n})=n+1$, $D(P_n) = 2$ for $n\geq 2$, $D(C_3)=D(C_4)=D(C_5)=3$ while $D(C_n)=2$ for $n\geq 6$ \cite{ac}.

A whole wealth of results on the subject has already been generated. Among many, we can only mention a few, only those that have essentially important results, or those that have introduced new indices based on distinguishing colorings. For a connected finite graph $G$, it was independently proved by Collins and Trenk \cite{Collins} and Klav\v{z}ar et al. \cite{klavzar} that $D(G) \leq \Delta + 1$, where $\Delta$ is the largest degree of $G$. Equality holds if and only if $G$ is a complete graph $K_{\Delta+1}$, a balanced complete bipartite graph $K_{\Delta,\Delta}$, or $C_5$. Collins and Trenk \cite{Collins} also mixed the concept of distinguishing colorings with proper vertex colorings to introduce the \emph{distinguishing chromatic number $\chi_{D}(G)$} of a graph $G$. It is defined as the minimum number of colors required to properly color the vertices of $G$ such that this coloring is only preserved by the trivial automorphism. They also showed that, for a finite connected graph $G$, we have $\chi_D (G) \leq 2\Delta(G)$ and that equality holds only if $G$ is isomorphic to $K_{\Delta,\Delta}$ or $C_6$. 

Symmetry breaking can also happen via other kinds of graph colorings. An analogous index for an edge coloring, namely the \emph{distinguishing index} $D'(G)$, has been introduced by Kalinowski and Pil\'sniak in \cite{Kalinowski} as the minimum number of the required colors in an asymmetric edge-coloring of a connected graph $G\not\simeq K_2$. Moreover, they showed that $D'(G) \leq \Delta(G)$ for a finite connected graph $G$, unless $G$ is isomorphic to $C_3$, $C_4$ or $C_5$. Another analogous index is introduced by Kalinowski, Pil\'sniak and Wo\'zniak in \cite{Rafal}; the \emph{total distinguishing number} $D''(G)$ is the minimum number of   required colors in an asymmetric total coloring of $G$.

To generalize some results from the finite case to the infinite ones, Imrich, Klav\v{z}ar and Trofimov \cite{Wilfried} considered the distinguishing number for infinite graphs. They showed that for an infinite connected graph $G$ we have $D(G) \leq \mathfrak{n}$, where $\mathfrak{n}$ is a cardinal number such that the degree of any vertex of $G$ is not greater than $\mathfrak{n}$.  Most symmetry breaking concepts have their relative counterparts in the infinite case, however, there are some (such as the \emph{Infinite Motion Conjecture}) that arise only when we consider infinite graphs. As an instance, one can take a look at \cite{comparison} by Imrich et al. which contains comparisons of some distinguishing indices of connected infinite graphs.

It was also interesting to know the distinguishing number for product graphs. For example, Bogstad and Cowen \cite{bc} showed that for $k \geq 4$, every hypercube $Q_k$ of dimension $k$, which is the Cartesian product of $k$ copies of $K_2$, is $2$-distinguishable. It has also been shown by Imrich and Klav{\v{z}}ar in \cite{ik} that the distinguishing number of Cartesian powers of a connected graph $G$ is equal to two except for $K_2^2, K_3^2, K_2^3$. Meanwhile, Imrich, Jerebic and Klav\v{z}ar \cite{ijk} showed that Cartesian products of relatively prime graphs whose sizes are close to each other can be distinguished with a small number of colors. Moreover, Estaji et al. in \cite{eikpt} proved that for every pair of connected graphs $G$ and $H$ with $|H| \leq |G| < 2^{|H|} - |H|$, we have $D(G \Box H) \leq 2$. Gorzkowska, Kalinowski and Pil\'{s}niak proved a similar result for the distinguishing index of the Cartesian product \cite{gkp}.

The lexicographic product was a subject of symmetry breaking via vertex and edge coloring by Alikhani and Soltani in \cite{saeed}, where they showed that under some conditions on the automorphism group of a graph $G$, we have $D(G)\leq D(G^k)\leq D(G)+k-1$, where $G^k$ is the $k$th lexicographic power of $G$, for any natural number $k$. As well, they showed that if $G$ and $H$ are connected graphs, then  $D(H)\leq D(G\circ H)\leq |V(G)|\cdot D(H)$.

Coloring is not the only mean of symmetry breaking in graphs. For example, one might break the symmetries of a graph  via a more general tool such as vertex partitioning. Ellingham and Schroeder introduced \emph{distinguishing partition} of a graph as a partition of the vertex set that is preserved by no nontrivial automorphism \cite{Ell}. Here, unlike coloring, some graphs have no distinguishing partition. Anyhow, for a graph $G$ that admits a distinguishing partition, one may think of the  minimum number of required cells in a distinguishing partition of the vertex set. Here, we show this index  by $DP(G)$.

In this paper, we  introduce some further  indices related  to symmetry breaking of graphs by studying the number of non-equivalent distinguishing colorings of a graph with $k$ colors and some other similar quantities. This is motivated by  the problem of evaluating  the distinguishing number of a lexicographic product or the $X$-join of some graphs, which we consider in Section~\ref{lexico}.

The paper is organized as follows. In Section~\ref{dist-col}, we consider  the number of non-equivalent distinguishing colorings of a graph $G$ with (exactly) $k$ colors, namely $\Phi_k (G)$ (and $\varphi_k (G)$) and, we calculate these indices for some simple types of graphs. Afterwards in Section \ref{threshold}, we introduce the \emph{distinguishing threshold} as a dual index to the distinguishing number. It is shown that calculations of some indices introduced in Sections \ref{dist-col} and \ref{dist-part} are easier, in some cases, when we know the distinguishing threshold. Moreover, in Section~\ref{dist-part}, we consider the number of non-equivalent \emph{distinguishing coloring partitions} of a graph $G$ with (exactly) $k$ cells, namely  $\Psi_k (G)$ (and $\psi_k (G)$) and, we calculate them in some special cases. Additionally in Section~\ref{dist-part}, some other auxiliary indices are also introduced. Then, we present an application of one of the indices introduced here, namely $\Phi_k (G)$, in Section~\ref{lexico}. We finally conclude  the  paper by shedding some lights on the future investigations  in Section~\ref{conclusion}.

Here, we use the standard notation and terminology of graph theory, which can  be found in~\cite{Diestel}. We only remind that the set of neighbors of a vertex $v$ in G is denoted by $N(v)$, while $N[v]$ stands for the set $N(v)\cup \{ v \}$.

\section{Non-equivalent Distinguishing Colorings}\label{dist-col}

Two colorings $c_1$ and $c_2$ of a graph $G$ are called \emph{equivalent} if there is an automorphism $\alpha$ of $G$ such that $c_1 (v) = c_2 (\alpha(v) )$ for all $v\in V (G)$.

The number of non-equivalent distinguishing colorings of a graph $G$ with $\{1,\ldots,k\}$ as the set of admissible   colors is shown by $\Phi_k (G)$, while the number of non-equivalent $k$-distinguishing colorings of a graph $G$ with $\{1,\ldots,k\}$ as the set of colors is shown by $\varphi_k (G)$. When $G$ has no distinguishing colorings with exactly $k$ colors, we put $\varphi_k (G) =0$. It is also clear that $\Phi_{D(G)} (G) = \varphi_{D(G)} (G)$. Moreover, it is straightforward to show that $$\Phi_k (G)= \sum_{i=D(G)}^{k} \binom{k}{i}\varphi_i (G).$$
Note also that $\varphi_k (K_n )$ is nonzero only when $k=n$, for which we know $\varphi_n (K_n)=1$. It is easy to prove that for $n\geq 2$ and $k\geq n$,
 $$\Phi_k (K_{n})=\binom{k}{n}.$$

In the following two theorems, we give some recursive formulas for $\Phi_k (P_n)$ and $\varphi_k (P_n)$. An interested reader can find some explicit formulas for these indices (and $\Phi_k (C_n)$ and $\varphi_k (C_n)$ as well), in the \href{https://oeis.org}{Online Encyclopedia of Integer Sequences} under the relevant sequence number (see \nameref{App_A}).

\begin{theorem} \label{Phi_P_n}
	For $n=4, 5, \ldots$, we have $$\Phi_k (P_n) = \binom{k}{2} k^{n-2}+k\Phi_k (P_{n-2}),$$ while $\Phi_k (P_{2}) = \binom{k}{2}$ and $\Phi_k (P_{3})= k\binom{k}{2}$.
\end{theorem}
\begin{proof}
	The proof is clear for $n=2,3$. For $n\geq 4$, we know that the two end-vertices of $P_n$ either have different colors or the same color. If they are different in colors, we can pick the two colors in $\binom{k}{2}$ different ways, and, the internal vertices can have any possible $k^{n-2}$ combinations, because the only non-trivial automorphism of $P_n$ has to map its end vertices onto one another. When the two end-vertices  have the same color, we can choose their color in $k$ different ways, but to be sure that the coloring is distinguishing, the remaining path of length $n-2$ must be distinguishingly colored in $\Phi_k (P_{n-2})$ ways.  
\end{proof}

In the next result, we make use of the well-known fact that  the number of surjective functions from a set of $n$ elements to a set of $k$ elements is $k!\stirling{n}{k}$ where $\stirling{n}{k}$ is the Stirling number of the second type.
\begin{theorem}\label{varphi_P_n}
	For $n=4, 5, \ldots$, we have $$\varphi_k (P_n) = k\big(\varphi_k (P_{n-2}) + \varphi_{k-1} (P_{n-2})\big)+ \binom{k}{2} \Big( (k-2)! \stirling{n-2}{k-2} +2(k-1)! \stirling{n-2}{k-1}+k!\stirling{n-2}{k} \Big).$$
\end{theorem}
\begin{proof}
	With the same method of counting to the proof of Theorem \ref{Phi_P_n}, we know that either the colors of the two ends of $P_n$ are the same or they are different. If they are the same, we can pick this color in $k$ different ways. Since the coloring has to be distinguishing, the internal path of length $n-2$ has to be colored distinguishingly with either all $k$ colors or the remaining $k-1$ colors. Therefore in this case we have $ k\big(\varphi_k (P_{n-2}) + \varphi_{k-1} (P_{n-2})\big)$ non-equivalent distinguishing colorings for $P_n$.
	
	When the two end-vertices of $P_n$ have different colors, any arbitrary coloring of internal vertices makes the resulting coloring a distinguishing one.  We can pick the two colors for the end vertices in $\binom{k}{2}$ ways while the rest of colors must be presented on the  internal vertices. Thus, there are four different possibilities; either only the rest of $k-2$ colors are used to color the  internal vertices of $P_n$, or these $k-2$ colors are used along with one of the two colors of the end vertices (two different possibilities), or all the $k$ colors are presented on the internal vertices of $P_n$. In each case we must count the number of surjective functions from the set of available colors to the set of internal vertices. Therefore in this case we have $\binom{k}{2} \Big( (k-2)! \stirling{n-2}{k-2} +2(k-1)! \stirling{n-2}{k-1}+k!\stirling{n-2}{k} \Big)$ non-equivalent distinguishing colorings for $P_n$ with exactly $k$ colors.
\end{proof}

Additionally, we calculate these indices for the complete bipartite graph $K_{n,n}$. Note that $D(K_{n,n})=n+1$.
\begin{theorem}
	For $n=2,3,\ldots$ we have $$\Phi_k (K_{n,n})= \frac{1}{2}\binom{k}{n}\big(\binom{k}{n}-1\big),$$ and for $n+1 \leq k \leq 2n$ we have $$\varphi_k (K_{n,n})= \frac{1}{2}\binom{k}{n}\binom{n}{k-n}.$$
\end{theorem}

\begin{proof}
The proof is easy for the first assertion since to color $K_{n,n}$ distinguishingly, one should  color the vertices of each part different from other vertices within that part and the two parts of $K_{n,n}$ have to be colored differently. Since we have chosen $n$ colors out of $k$ colors to color the vertices of a part of $K_{n,n}$, we cannot color the other part by the same pallet of colors. Moreover, transposing  the first and the second part makes the two resulting colorings equivalent. Therefore, we arrive at a conclusion for the first assertion.

For the second assertion,  like the first one, we color the vertices of one part by choosing $n$ colors out of $k$ ones, then we  use the remaining $k-n$ colors for the vertices of the other part. For the rest of the vertices, i.e. the $2n-k$ vertices in  the other part, we must choose colors from the first set of $n$ colors in $\binom{n}{2n-k}=\binom{n}{n-2n+k}=\binom{n}{k-n}$ ways. Again, by transposing  the first and the second part, the two resulting colorings are equivalent. This makes the second assertion evident.
\end{proof}

In \nameref{App_A},  we have presented the parameters $\Phi_k$ and $\varphi_k$ for the graphs $P_n$ and $C_n$, for $n,k=2,\ldots,10$.

It might seem easy to calculate $\Phi_k (G)$ and $\varphi_k (G)$ when $G$ is a path or a cycle. However, the calculations are not easy in the general case. Even for a computer algebra system, it might take very long to count the number of non-equivalent distinguishing colorings of a symmetric graph $G$ on $n$ vertices when $n\geq 10$. Anyhow, even when $n$ is large, for some $k$ the calculations are much easier.

Let $G$ be a graph on $n$ vertices. Assume that we desire to distinguishingly color the vertices of $G$ with exactly $n$ colors. Then every vertex must receive a color different from the others which gives rise to $n!$ colorings. However,  in order to count  non-equivalent colorings,   we should consider the colorings modulo the automorphism group  of $G$. Therefore, we have  the following  result, which coincides with Theorem \ref{stirling_general_graph} in the next section.
\begin{proposition}\label{k=n}
	For any graph $G$ on $n$ vertices, we have 
	$$\varphi_n(G)=\frac{n!}{|\aut(G)|}.\qed$$
\end{proposition}

This result motivates us to ask whether there are  numbers $k\leq n$ such that any $k$-coloring of a graph on $n$ vertices is distinguishing.  We will consider this problem in Section~\ref{threshold}.

\section{Distinguishing Threshold}\label{threshold}

For any graph $G$, we define the \textit{distinguishing threshold} $\theta(G)$ to be the minimum number $t$ such that for any $k\geq t$, any arbitrary coloring of $G$ with $k$ colors is distinguishing.
For example $\theta(K_n)=\theta(\overline{K_n})=n$ and $\theta(K_{m,n})=m+n$. Note, also, that for an asymmetric graph $G$, we have $\theta(G)=D(G)=1$. Moreover, we always have $\theta(G)\geq D(G)$.

Let $G$ be a graph on $n$ vertices. If any two distinct vertices $u,v\in G$ have different set of neighbors other than themselves, then any $(n-1)$-coloring of $G$ has to be distinguishing because in this case, no two vertices with the same color can be mapped to each other by a non-trivial color-preserving automorphism. Conversely, if there are two vertices $u$ and $v$ such that $N(v)\setminus \{u\} = N(u)\setminus \{v\}$, then any $(n-1)$-coloring of $G$ which assigns the same color to these two vertices, is not distinguishing. 
From this, we observe the following.

\begin{lemma}\label{theta_of_G}
	For any graph $G$ on $n$ vertices, $\theta(G)\leq n-1$ if and only if  $N(v)\setminus \{u\} \neq N(u)\setminus \{v\}$, for all distinct vertices $u$ and $v$.\qed
\end{lemma}

For the cases of paths and cycles we can calculate the distinguishing threshold.

\begin{proposition}\label{theta_of_P_n}
For any $n\geq 2 $ we have
	\[ \theta(P_n)=\lceil\frac{n}{2}\rceil+1. \]
\end{proposition}

\begin{proof}
The automorphism group of $P_n$ induces $\lceil \frac{n}{2} \rceil$ orbits on its  vertices, while each orbit contains at most two vertices. By the pigeonhole principle, any combination of $\lceil\frac{n}{2}\rceil+1$ colors on $n$ vertices of $P_n$ breaks at least one orbit. Since with  $\lceil\frac{n}{2}\rceil$ colors, there is a non-distinguishing coloring, we must have $\theta(P_n)=\lceil\frac{n}{2}\rceil+1$.
\end{proof}
\begin{proposition}\label{theta_of_C_n}
	For any $n\geq 3 $ we have
	\[ \theta(C_n)=\lfloor\frac{n}{2}\rfloor+2 \]
\end{proposition}
\begin{proof}
To show that $\theta(C_{n})\geq \lfloor\frac{n}{2}\rfloor+2$, we present an $(\lfloor\frac{n}{2}\rfloor+1)$-coloring of an $n$-cycle with the vertex set $\{ v_{1}, \ldots, v_n \}$ which is not a distinguishing coloring. If $n$ is even, then color the  vertices $v_{1},\ldots , v_{\lfloor\frac{n}{2}\rfloor+1}$ by colors $1,\ldots , \lfloor\frac{n}{2}\rfloor+1$, and color the vertices $v_{\lfloor\frac{n}{2}\rfloor+2}, \ldots , n$ by $\lfloor\frac{n}{2}\rfloor, \ldots , 2$, respectively. Similarly, if $n$ is odd, color the vertices $v_{1},\ldots , v_{\lfloor\frac{n}{2}\rfloor+1}$ by colors $1,\ldots , \lfloor\frac{n}{2}\rfloor+1$, and color $v_{\lfloor\frac{n}{2}\rfloor+2}, \ldots , n$ by $\lfloor\frac{n}{2}\rfloor+1, \ldots , 2$, respectively. This coloring is not  distinguishing  as it cannot break the reflection symmetry that fixes $v_1$ and maps $v_2$ on $v_n$. Figure \ref{c_n} illustrates this coloring for $C_8$ and $C_9$.

It remains to show that for $k\geq \lfloor\frac{n}{2}\rfloor + 2$, every $k$-coloring of $C_n$ is distinguishing. When $n$ is odd, coloring the vertices of $C_n$ with $k$ colors results in  at least three colors that are used only once. Similarly, when $n$ is even, coloring the vertices of $C_n$ with $k$ colors results in at least four colors that are used only once. Hence, in any case, for any coloring of $C_n$, with $k$ colors, there are at least three colors which are used only once. Now consider vertices $v_1 , v_2$ and $v_3$ whose colors appeared only once in a $k$-coloring and suppose that $P$ is the only path in $C_n$ from $v_1$ to $v_2$ that does not contain $v_3$. Any color-preserving automorphism $\alpha$ of $C_n$ must map $P$ onto itself, which means that $\alpha$ is the identity. This completes the proof.
\end{proof}

\begin{figure}[h!]
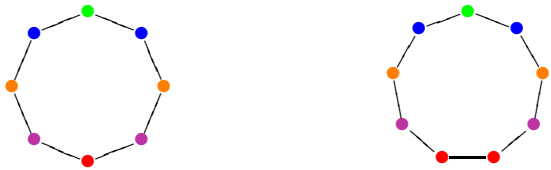
 
	\[ \xygraph{
		!{<0cm,0cm>; <1cm,0cm>:<0cm,1cm>::}
		%vertices
		!{(0,1)}*{\textcolor[rgb]{0.00,1.00,0.00}{\bullet}}="v1" !{(0,-1)}*{\textcolor[rgb]{1.00,0.00,0.00}{\bullet}}="v5"
		!{(0.7071,0.7071)}*{\textcolor[rgb]{0.00,0.00,1.00}{\bullet}}="v2" !{(-0.7071,-0.7071)}*{\textcolor[rgb]{0.74,0.20,0.64}{\bullet}}="v6" 
		!{(1,0)}*{\textcolor[rgb]{1.00,0.49,0.00}{\bullet}}="v3" !{(-1,0)}*{\textcolor[rgb]{1.00,0.49,0.00}{\bullet}}="v7"
		!{(0.7071,-0.7071)}*{\textcolor[rgb]{0.74,0.20,0.64}{\bullet}}="v4" !{(-0.7071,0.7071)}*{\textcolor[rgb]{0.00,0.00,1.00}{\bullet}}="v8"
		!{(5,1)}*{\textcolor[rgb]{0.00,1.00,0.00}{\bullet}}="v9" !{(5.64279,0.76604)}*{\textcolor[rgb]{0.0,0.0,1.0}{\bullet}}="v10"
		!{(5.98481,0.17364)}*{\textcolor[rgb]{1.00,0.49,0.00}{\bullet}}="v11" !{(5.86602,-0.5)}*{\textcolor[rgb]{0.74,0.20,0.64}{\bullet}}="v12" 
		!{(5.34202,-0.93969)}*{\textcolor[rgb]{1.00,0.00,0.00}{\bullet}}="v13" !{(4.65798,-0.93969)}*{\textcolor[rgb]{1.0,0.00,0.00}{\bullet}}="v14"
		!{(4.13397,-0.5)}*{\textcolor[rgb]{0.74,0.20,0.64}{\bullet}}="v15" !{(4.01519,0.17364)}*{\textcolor[rgb]{1.00,0.49,0.00}{\bullet}}="v16"
		!{(4.35721,0.76604)}*{\textcolor[rgb]{0.0,0.0,1.0}{\bullet}}="v17"
				%edges
		"v1"-@[black] "v2" "v2"-@[black] "v3" "v3"-@[black] "v4" 
		"v4"-@[black] "v5" "v5"-@[black] "v6"
		"v6"-@[black] "v7" "v7"-@[black] "v8" "v1"-@[black] "v8"
		"v9"-@[black] "v10" "v10"-@[black] "v11" "v11"-@[black] "v12" 
		"v12"-@[black] "v13" "v13"-@[black] "v14"
		"v14"-@[black] "v15" "v15"-@[black] "v16"
		"v16"-@[black] "v17" "v9"-@[black] "v17"
	} \]
	\caption{Non-distinguishing colorings for $C_8$ and $C_9$ with  5 distinct colors.}\label{c_n}
\end{figure}

It, therefore, seems an interesting problem on its own to calculate the distinguishing  threshold of various families of graphs. As an example, one might ask this question for the Petersen graph $P$. In fact, this graph is a member of a well-known family of graphs, namely, the  \textit{Kneser} graphs. Let $0\leq k\leq n/2$. Then, the Kneser graph $K(n,k)$ is a graph whose vertex set is the set of all $k$-subsets of $\{1,\ldots,n\}$ where two vertices are adjacent if their corresponding sets do not intersect. It is easy to see that $K(n,k)$ is a vertex-transitive graph on ${n\choose k}$ vertices  with valency ${n-k\choose k}$, and that  $P=K(5,2)$. To study more on Kneser graphs see, for instance,~\cite{Godsil}.
 
We now evaluate the distinguishing threshold of the Kneser graphs $K(n,2)$.

\begin{proposition}\label{theta_of_kneser}
For any $n\geq 5$, we have $\theta(K(n,2))=\frac{1}{2}(n^2-3n+6)$. In particular, if $P$ is the Petersen graph, then   $\theta (P) =8$.
\end{proposition}
\begin{proof}
Assume that $k\geq\frac{1}{2}(n^2-3n+6)$ and, to get a  contradiction, suppose that  there is a $k$-coloring of $K(n,2)$ which is not distinguishing. This implies that there is a color-preserving automorphism  $\alpha$ such that $\alpha(u)=v$, for some distinct vertices $u$ and $v$. Let $R$  be the set of common neighbors of $u$ and $v$. Thus 
\[ r=|R|={n-4+i\choose 2}, \]
where $i\in \{0,1\}$ is the possible size of the intersection of the corresponding $2$-sets of $u$ and $v$. Furthermore, let 
$S=N(u)\setminus N[v]$
 and 
  $S'=N(v)\setminus N[u]$. Therefore, we have 
\[s=|S|=|S'|={n-2\choose 2}-r+i-1.\]
Moreover, let 
\[ X=V(K(n,2))\setminus \left(\{u,v\}\cup R\cup S\cup S'\right).\]
Hence 
\[ x=|X|={n\choose 2}-(2+r+2s). \]
Note that, since we have $\alpha(N(u))=N(v)$, the color-pallet of $N(u)$ must be the same as that of $N(v)$. Thus, the number of colors used for coloring the vertices in $\{u,v\}\cup R\cup S\cup S'$ is at most $1+r+s$. Consequently, we have to color the vertices in $X$ with $k-(1+r+s)$ colors. But this is impossible because $k-(1+r+s)>x$. The reason is as follows:
\begin{align*}
 k-(1+r+s)-x&=k-1-r-s-{n\choose 2} +2+r+2s\\
 &=k-{n\choose 2}+1+s\\
 &\geq  \frac{1}{2}(n^2-3n+6) -{n\choose 2}+1 + {n-2\choose 2} - {n-4+i\choose 2} + i-1  \\
 & \geq \frac{1}{2}(n^2-3n+6)  -{n\choose 2}+1 + {n-2\choose 2} - {n-3\choose 2}>0
 \end{align*}
 
Now, suppose that $k= \frac{1}{2}(n^2-3n+4)=\frac{1}{2}(n^2-3n+6)-1$. We show that there is a $k$-coloring of $K(n,2)$ which is not distinguishing. Color the vertices $\{3,1 \}$ and $\{3,2\}$ by the color $1$, the vertices $\{4,1 \}$ and $\{4,2\}$ by the color $2$, $\ldots$, and the vertices $\{n,1 \}$ and $\{n,2\}$ by the color $n-2$, while the  other ${n\choose 2}-2n+4$ vertices receive the remaining $k-n+2$ colors. Note that 
\[{n\choose 2} -2n+4 =k-n+2.\]
Now, consider the mapping $(\{3,1\},\{3,2\} )(\{4,1\},\{4,2\} )(\{5,1\},\{5,2\} )\cdots(\{n,1\},\{n,2\})$ which is a nontrivial color-preserving automorphism of $K(n,2)$. 
\end{proof}

The next result reveals the importance of the distinguishing threshold in counting the number of non-equivalent distinguishing colorings.

\begin{theorem}\label{stirling_general_graph}
	Let $G$ be a graph  on $n$ vertices. For any $k\geq \theta(G)$ we have
	\[ \varphi_k(G)=k! \stirling{n}{k}/|\aut(G)|. \]
\end{theorem}
\begin{proof}
	When $k\geq \theta(G)$, any $k$-coloring of $G$ is distinguishing. Therefore, to calculate $\varphi_k (G)$, we must only count the number of non-equivalent $k$-colorings. We know that the total number of $k$-colorings of $G$ is equal to the number of surjective functions from the set of vertices to the set of colors. As it is noted just before Theorem \ref{varphi_P_n}, this number is $k! \stirling{n}{k}$. Furthermore, for any automorphism $\alpha \in \mathrm{Aut}(G)$, the image of a $k$-coloring $f$ under $\alpha$ is equivalent to $f$. Consequently, the result holds.
\end{proof}

Using Theorem~\ref{stirling_general_graph} and Propositions~\ref{theta_of_P_n} and~~\ref{theta_of_C_n} the following results follow immediately.

\begin{corollary}\label{stirling_P_n}
	Let $n\geq 2$. For any $k\geq \lceil\frac{n}{2}\rceil+1$ we have
	\[ \varphi_k(P_n)=k! \stirling{n}{k}/2.\qed \]
\end{corollary}

\begin{corollary}\label{stirling_C_n}
	Let $n\geq 3$. For any $k\geq \lfloor\frac{n}{2}\rfloor+2$ we have
	\[ \varphi_k(C_n)=k! \stirling{n}{k}/2n.\qed \]
\end{corollary}

Furthermore, it is immediate to observe that, for any $n\geq 4$, $ \varphi_{n-1}(P_n)=\frac{1}{4}(n-1)n!$ and that for any $n\geq 5$, $\varphi_{n-1}(C_n)=\frac{1}{4}(n-1)(n-1)!$, which agree with the tables in \nameref{App_A}.

\section{Non-equivalent Distinguishing Partitions}\label{dist-part}
In this section, we turn our attention to the case of different distinguishing partitions of graphs. Let $G$ be a graph and let $P_1$ and $P_2$ be two partitions of the vertices of $G$. We say $P_1$ and $P_2$ are \textit{equivalent} if there is a non-trivial automorphism of $G$ which maps $P_1$ onto $P_2$. The number of non-equivalent partitions of $G$, with at most $k$ cells, is called the \textit{partition number} of $G$ and is denoted by $\Pi_k(G)$. 

Meanwhile, a \emph{distinguishing coloring partition} of a graph $G$ is a partition of the vertices of $G$ such  that it induces a distinguishing coloring for $G$. Note that the minimum number of cells required for such a partition equals the distinguishing number $D(G)$. The number of non-equivalent  distinguishing coloring partitions of a graph $G$ with at most $k$ cells is shown by $\Psi_k (G)$, while the number of non-equivalent distinguishing coloring partitions of a graph $G$ with exactly $k$ cells is shown by $\psi_k (G)$. It is not difficult to observe that 
\[\Psi_k(G)=\sum_{j\leq k }\psi_j(G),\]
and that $\psi_n(G)=1$, for any graph $G$ on $n$ vertices. 
In what follows, we deal with $\psi_k(G)$ and present some calculations where $G$ is a path or a cycle. We start with the following observation which states how $\psi_2(P_n)$ is related to $\varphi_2(P_n)$.

\begin{proposition}\label{psi_vs_phi}
	For any $n\geq 1$, $\psi_2(P_{2n+1})=\frac{1}{2}\varphi_2(P_{2n+1})$ and $\psi_2(P_{2n})=\frac{1}{2}\varphi_2(P_{2n})+2^{n-2}$.
\end{proposition}
\begin{proof}
First we consider the path $P_{2n+1}$. It is evident that any distinguishing coloring partition with two cells induces two distinguishing colorings. Furthermore, since only one of the two cells contains the middle vertex of $P_{2n+1}$, these two colorings are non-equivalent; this proves the first part. 

To see the next part, we note that $\varphi_2(P_{2n})$ counts two different types of distinguishing coloring partitions: it counts type~1, which consists of the partitions in which swapping the two colors makes the resulting colorings equivalent, only once, while it counts type~2, which are the remaining partitions, twice. Hence, if we add the number of distinguishing coloring partitions  of type~1 to $\varphi_2(P_{2n})$, then  every distinguishing coloring partitions of $P_{2n}$ is counted twice. On the other hand, it is not hard to see that the number of distinguishing coloring partitions of type~1 is equal to $2^n/2=2^{n-1}$. Therefore,
\[ \psi_2(P_{2n})=\frac{\varphi_2(P_{2n})+2^{n-1}}{2}, \]
which completes the proof.
\end{proof}

An immediate consequence of Proposition~\ref{psi_vs_phi} is that $\varphi_2(P_n)$ is always an even  number.  We can, furthermore, calculate $\psi_k(P_{n})$ in the case $k=n-1$.

\begin{proposition}\label{psi_k=n-1_for_Pn}
	Let $n\geq 2$. We have $\psi_{n-1}(P_n)=\lfloor\frac{n^2}{4}\rfloor$.
\end{proposition}
\begin{proof}
	As there are $n-1$ different cells, any partition  of the vertices  of $P_n=v_1v_2 \cdots v_n$ will result in exactly one pair of vertices in the same part of the partition. So it suffices to count the number of different ways to choose two vertices such that including them in one cell and all the other vertices in singleton cells, results in a distinguishing coloring partition such that no two such partitions are mapped to each other using the  non-trivial automorphism of $P_n$. First, assume $n$ is even.  We split $P_n$ to two halves: $A=\{v_1,\ldots,v_{\frac{n}{2}}\}$ and $B=\{v_{\frac{n}{2}+1},\ldots,v_{n}\}$. There are two non-equivalent cases: (a) the two vertices are chosen from $A$, and (b) one is chosen from $A$  and the other one is chosen from $B$. The case (a) contains ${n/2\choose 2}$ ways. On the other hand, case (b), in turn, contains the following subcases:  if one chooses $v_1$, then there are $|B|=\frac{n}{2}$ choices for the second vertex; if one chooses $v_2$, then there are $\frac{n}{2}-1$ choices for the second vertex (note that  the case $(v_2,v_n)$  is equivalent to   $(v_1,v_{n-1})$ which has already been counted). Continuing this argument, we will have 
	\[ \frac{n}{2}+(\frac{n}{2}-1)+\cdots+1=\frac{n^2}{8}+\frac{n}{4} \]
	choices in case (b). Hence, the total number of ways is $n^2/4$ and the result follows.
	
	In the case where $n$ is odd, we set $A=\{v_1,\ldots,v_{\frac{n-1}{2}}\}$ and $B=\{v_{\frac{n+1}{2}+1},\ldots,v_{n}\}$. Similar to the even case above,  there are $\frac{(n-1)^2}{4}$ choices for the pairs  to belong to the same cell of the partition, where either the pair is chosen from $A$ or one vertex from $A$ and the other one from $B$. In addition, there are $\frac{n-1}{2}$ further non-equivalent partitions  in this case, in which the 2-vertex cell consists of a vertex of  $A$ along  with the middle  vertex $v_{\frac{n+1}{2}}$. We conclude that the total number of choices is $(n^2-1)/4$, which completes the proof.
\end{proof}

In the next proposition, we consider the same problem as in Proposition~\ref{psi_k=n-1_for_Pn}  for the case of cycles in which we make use of the distinguishing threshold of $C_n$.  
\begin{proposition}\label{psi_k=n-1_for_Cn}
	Let $n\geq 3$. We have $\psi_{n-1}(C_n)=0,1$, if $n=3$ and $n=4$, respectively, and $\psi_{n-1}(C_n)=\lfloor \frac{n}{2}\rfloor$, if $n\geq 5$. 
\end{proposition}
\begin{proof}
	It is easy to check the result in the small cases $n=3,4$ directly. We, therefore, consider the case $n\geq 5$. In this case, $n-1\geq \lfloor n/2\rfloor+2=\theta(C_n)$ and, according to Proposition~\ref{theta_of_P_n}, any coloring of $C_n$ with $n-1$ colors is a distinguishing coloring. Thus, it is sufficient to count the number of non-equivalent partitions of the vertices  $\{v_1, v_2, \cdots, v_n\}$ of $C_n$ into $n-1$ cells, such that no two such partitions are mapped to each other using an automorphism of $C_n$. It is not hard to see that there is a one-to-one correspondence between the family of such partitions and the set of all possible distances in $C_n$. In other words, the only such partitions are the ones including the cells $\{v_1,v_2\}$, $\{v_1,v_3\}$, $\ldots$, $\{v_1,v_{\lfloor\frac{n+1}{2}\rfloor}\}$. Therefore the result follows.
\end{proof}

We note that $\psi_k(G)=\Psi_k(G)-\Psi_{k-1}(G)$, for any graph $G$; in other words, in order to calculate $\psi_k(G)$, we can use $\Psi_k(G)$ if we already know the latter index of $G$. In the next theorem, we calculate $\Psi_k(P_n)$. To do so, we need the following notation. Recall that, for a given graph $G$, a distinguishing partition is a partition of the vertex set of $G$ such that no nontrivial automorphism of $G$ can preserve it (see~\cite{Ell}) and  that $DP(G)$ is the minimum number of cells in a distinguishing partition of $G$. It is evident that if $G$ admits a distinguishing partition, then $DP(G)\geq D(G)$. For $n\geq 3$, the path $P_n$ admits a distinguishing partition and $DP(P_n)=2$. 

 We define $\Xi_k(G)$ to be the number of non-equivalent distinguishing partitions of $G$ with at most $k$ cells. Correspondingly, $\xi_k(G)$ denotes the  number of non-equivalent distinguishing partitions of $G$ with exactly $k$ cells. It is not hard to see that
 \[ \Xi_k(G)=\sum_{j\leq k} \xi_j(G). \]
Note that, if $G$ does not admit a distinguishing partition with exactly $k$ cells, then $\xi_k(G)=0$ and that $\xi_k(G)=\Xi_k(G)-\Xi_{k-1}(G)$.

\begin{theorem}\label{xi}
Let $n\geq 2$. For any $k\geq 2$ we have
\[\Psi_k(P_n)=\Pi_k(P_n)-\Pi_k(P_{\lceil \frac{n}{2}\rceil})- \Xi_k(P_{\lceil \frac{n}{2}\rceil})\]
\end{theorem}
\begin{proof}
In order to count $\Psi_k(P_n)$, the number of non-equivalent distinguishing coloring partitions of $P_n$ with at most $k$ cells, we should subtract the number of non-distinguishing partitions from  the total number of non-equivalent partitions of $P_n$ with at most $k$ cells, i.e. $\Pi_k(P_n)$. 
Note that if a partition of $P_n$ is non-distinguishing, then its restriction to $P_{\lceil \frac{n}{2}\rceil}$, is either a non-distinguishing or a distinguishing partition. The number of the partitions of the former type is 
$ \Pi_k(P_{\lceil \frac{n}{2}\rceil}) - \Xi_k(P_{\lceil \frac{n}{2}\rceil})$,
while the number of the partitions of the latter type is 
$ \Xi_k(P_{\lceil \frac{n}{2}\rceil})$

Let $\alpha$ and $\beta$ be the non-trivial automorphisms of $P_n$ and $P_{\lceil \frac{n}{2}\rceil}$, respectively. If a partition $\pi=\{\pi_1,\ldots,\pi_r\}$ ($r\leq k$) of  $P_{\lceil \frac{n}{2}\rceil}$  is non-distinguishing, i.e. $\beta(\pi)=\pi$, then the lifted partition
\[\pi'=\{\pi_1\cup \alpha(\pi_1),\ldots,\pi_r\cup \alpha(\pi_r) \}\]
 of $P_n$ is  non-distinguishing. However, if a partition $\sigma=\{\sigma_1,\ldots,\sigma_r\}$ ($r\leq k$) of  $P_{\lceil \frac{n}{2}\rceil}$  is distinguishing, i.e. $\beta(\sigma)\neq\sigma$, then the two  lifted partitions 
 \[\sigma'=\{\sigma_1\cup \alpha(\sigma_1),\ldots,\sigma_r\cup \alpha(\sigma_r) \}\quad\text{and}\quad  \sigma''=\{\beta(\sigma_1)\cup \alpha(\beta(\sigma_1)),\ldots,\beta(\sigma_r)\cup \alpha(\beta(\sigma_r)) \}\]
  of $P_n$ are  non-equivalent  non-distinguishing partitions. On the other hand, by the definition, the number of partitions $\sigma$ is equal to $\Xi_k(P_{\lceil \frac{n}{2}\rceil})$. Therefore the total number of non-equivalent non-distinguishing partitions of $P_n$ equals to
  \[\Pi_k(P_{\lceil \frac{n}{2}\rceil})- \Xi_k(P_{\lceil \frac{n}{2}\rceil})+2 \Xi_k(P_{\lceil \frac{n}{2}\rceil})=\Pi_k(P_{\lceil \frac{n}{2}\rceil})+ \Xi_k(P_{\lceil \frac{n}{2}\rceil}),\]
which completes the proof.
\end{proof}

Theorem~\ref{xi} provides a nice connection among $\Psi_k(P_n)$, $\Pi_k(P_n)$ and $\Xi_k(P_n)$.  See \nameref{App_B} for tables of these indices.

\section{Distinguishing Lexicographic Products}\label{lexico}

In this section we provide an important application of one of the indices introduced in this paper, namely $\Phi_k(G)$. We start by recalling some preliminaries to the topic of lexicographic product of graphs.

Let $X$ be a graph. The \emph{$X$-join of $\{ Y_x \vert x\in V(X)\}$}, is the graph $Z$ with $$V(Z)= \{ (x,y) : x \in X, y\in Y_x\}$$ and $$E(Z) = \{ (x,y)(x',y'): xx'\in E(X)\textnormal{ or else }x = x'\textnormal{ and } yy'\in E(Y_{x})\}.$$
Whenever, for all $x\in X$, we have $Y_x \simeq Y$, for a fixed graph $Y$, the graph $Z$ is called the \emph{lexicographic product of $X$ and $Y$} and we write $Z=X\circ Y$.

We remind the reader that in \cite{saeed}, some bounds on the distinguishing number of the lexicographic product of graphs have been presented. In  this section, we calculate the distinguishing number of a lexicographic product in a natural case.

Automorphism groups of the $X$-join of $\{ Y\}_x$ and lexicographic products were studied by Hemminger \cite{hem} and Sabidussi \cite{Sab}.  Hemminger defined natural isomorphisms of an $X_1$-join graph onto $X_2$-join graph as follows: let $Z_i$ be an $X_i$-join of $\{ Y_{ix}\}_{x\in X_i}$, $i = 1, 2$. Then a graph isomorphism $\mu$ of $Z_1$ onto $Z_2$ is called \emph{natural} if for each $x_1 \in X_1$ there is an $x_2 \in X_2$ such that $\mu(Y_{1x_1 })= Y_{2x_2}$. Otherwise $\mu$ is called \textit{unnatural}. He then characterized all the $X$-join graphs whose automorphism groups consists of all their natural automorphisms \cite{hem}. When the automorphisms of an $X$-join graph (or a lexicographic product) are all natural ones, it is easier to break them, as we do it here.

\begin{theorem}\label{D_lexico}
	Suppose that $X\circ Y$ represents the lexicographic product of two graphs $X$ and $Y$. Then $D(X\circ Y)=k$ where $k$ is the least integer that $\Phi_k (Y)\geq D(X)$, provided that all the automorphisms of $X\circ Y$ are natural.
\end{theorem}
\begin{proof}
		When all the automorphisms of $G= X\circ Y$ are natural ones, any distinguishing coloring of  $G$ has to break symmetries inside $Y_x$ for all $x\in V(X)$ and for each automorphism $\alpha$ of $X$ that maps $u$ on $v$ ($u\neq v$), the colorings of $Y_u$ and $Y_v$ must be non-equivalent. Therefore, whenever $G$ has a distinguishing coloring with $k$ colors, we must have $\Phi_k (Y)\geq D(X)$.
\end{proof}

By a similar argument, we find an upper bound for the distinguishing number of the $X$-join of a set of graphs $\{Y_x \}_{x\in X}$. To do so, we need some  notation in our argument. Let $f$ be a distinguishing coloring of $X$ with $D(X)$ colors. For $x \in X$, let $$C(x) =\{w \in X : f(x)\neq f(w)\textnormal{ and } Y_x \simeq Y_w \}\cup \{ x \}$$ and $$D_x = \mathrm{min}\{ k : \Phi_k (Y_x) \geq \vert C(x)\vert \}.$$ For each $x\in X$, we obviously have $\vert C(x)\vert \geq 1$ and $D_x \geq D(Y_{x})$. Moreover, put $$d_f = \mathrm{max}\{ D_x : x\in X \}.$$

\begin{theorem}\label{x-join}
	Let $Z$ be the $X$-join of $\{Y_x \}_{x\in X}$ whose automorphism group is the  set of natural automorphisms. Let $S$ be the set of all (non-equivalent) distinguishing colorings of $X$ with $D(X)$ colors. Then $D(Z)\leq \mathrm{min}\{ d_f : f\in S \}$.
\end{theorem}
\begin{proof}
	Let $f\in S$  and suppose that for each $x\in X$, we colored $Y_x$ distinguishingly.  For each $w \in C(x)\setminus \{x\}$, if $Y_x$ and $Y_w$ are colored non-equivalently, then we can guarantee that the resulting coloring of $Z$ is distinguishing. For such a coloring we do not need more than $d_f$ colors. Consequently, $D(Z) \leq d_f$.
\end{proof}

We must point this out to the reader that sometimes it is possible that $D(Z)$ becomes strictly less than $\mathrm{min}\{ d_f : f\in S \}$. For example, suppose that $X$ is a large cycle on $\{ v_1 , v_2 , \ldots , v_n \}$ as its set of vertices. Let $Y_1$ be an asymmetric tree on $m_1 \geq 7$ vertices and $Y_3$ be another asymmetric tree on $m_3 \geq 7$ vertices, and $Y_1 \not\simeq Y_3$. Suppose also that for $i=2,4,5,\ldots, n$, all $Y_i$s are isomorphic to $K_1$. Then, the $X$-join of $\{ Y_i \}_{i=1}^{n}$ is an asymmetric graph (whose distinguishing number equals to 1), while $\mathrm{min}\{ d_f : f\in S \} = 2$. However, the stated bound in Theorem \ref{x-join} is the best that can be  generally found about $D(Z)$, because it is attainable by the lexicographic product of two graphs (e.~g. $C_n \circ K_1$).

\section{Conclusion}\label{conclusion}

We have seen in Section \ref{lexico} that counting the number of non-equivalent distinguishing colorings,~$\Phi$, has an application in finding the distinguishing number of lexicographic products. Moreover, other indices have shown to have  deep interactions with each other and also with~$\Phi$. It should be noted that calculating these indices are not always easy. Even when the automorphism group is very small and simple, counting non-equivalent distinguishing colorings or distinguishing (coloring) partitions faces with several calculation obstacles. 

In the appendices, there are tables of the indices introduced in this paper for small paths and cycles. Considering these tables suggests that most of these indices are new generators of integer sequences.

To make calculations more comfortable, in some special cases, we introduced the notion of  distinguishing threshold in Section \ref{threshold}, which enables us to reduce the required calculations in a computer algebra system to an acceptable level. However, this index has the importance to be considered separately, as it is a dual to the distinguishing number; \textit{there is a distinguishing coloring with a number of colors greater than or equal to the distinguishing number while there is a non-distinguishing coloring with a number of colors less than the distinguishing threshold}. Among many good questions, one might consider this index for some families of graphs, or, study the distinguishing threshold for the Cartesian product.

We finally point out that one might consider infinite graphs or some notions other than distinguishing coloring  for defining the parameters that we have introduced in this paper; e.~g. distinguishing edge coloring, distinguishing total coloring, etc.

\section*{Acknowledgment}
We owe a great debt to professor Wilfried Imrich, who proposed several problems, which led to this paper, during his visit to Shiraz University.

\newpage

\section*{Appendix A}\label{App_A}

\begin{center} \textbf{Tables of $\Phi_k$ and $\varphi_k$ for small paths and cycles} \end{center}

\begin{table}[h]
	\small\hspace{-1cm}
	\label{tab_for_Phi_k_of_P_n}
	\begin{tabular}{|l||l|l|l|l|l|l|l|l|l|} \hline
	\backslashbox[1mm]{$n$}{$k$}&2&3&4&5&6&7&8&9&10\\ \hline\hline
	2&1&3&6&10&15&21&28&36&45\\ \hline
	3&2&9&24&50&90&147&224&324&450\\  \hline
	4&6&36&120&300&630&1176&2016&3240&4950\\  \hline
	5&12&108&480&1500&3780&8232&16128&29160&49500\\  \hline
	6&28 &351&2016&7750&23220&58653&130816&265356&499500\\ \hline
	7&56&1053&8064&38750&139320&410571&1046528&2388204&4995000\\  \hline
	8&120&3240&32640&195000&839160&2881200&8386560&21520080&49995000\\ \hline 9&240&9720&130560&975000&5034960&20168400&67092480&193680720&499950000\\ \hline 10&496&29403&523776&4881250&30229200&141229221&536854528&1743362676&4999950000\\\hline
	\end {tabular}
	
	\medskip
	\caption{Some values of $\Phi_k(P_n)$ (Sequence A293500)}
\end{table}

\vfill

\begin{table}[h]
	\centering\small
	\label{tab_for_Phi_k_of_C_n}
	\begin {tabular}{|l||l|l|l|l|l|l|l|l|l|} \hline
	\backslashbox[10mm]{$n$}{$k$}&2&3&4&5&6&7&8&9&10\\ \hline\hline
	3&    0   & 1 &       4 &    10&                20&         35&       56&     84&    120\\  \hline
	4&    0   &   3  &     15 &      45  &          105&        210&       378&      630&990\\  \hline
	5&     0   &  12   &      72     & 252  &        672&         1512&      3024&     5544&9504\\  \hline
	6&     1   &   37   &      266   &   1120  &      3515  &      9121&      20692&      42456&80565\\ \hline
	7&     2 &     117 &      1044     & 5270  &     19350  &       57627  &     147752  &     338364 &709290\\  \hline
	8&     6  &    333    &     3788    & 23475   &      102690   &      355446  &      1039248   &     2673810   &6222150\\ \hline 
	9&    14 &    975  &      14056    &  106950   &      555990   &       2233469   &    7440160  &     21493836  &55505550\\ \hline 
	10&   30  &   2712  &      51132    &   483504  &    3009426    &      14089488  &     53611992   &    174189024   &  499720518  \\\hline
	\end {tabular}
	
	\medskip
	\caption{Some values of $\Phi_k(C_n)$ (Sequence A309528)}
\end{table}

\pagebreak

\begin{table}[h]
	\centering\small
	\label{tab_for_phi_k_of_P_n}
	\begin {tabular}{|l||l|l|l|l|l|l|l|l|l|} \hline
	\backslashbox{$n$}{$k$}&2&3&4&5&6&7&8&9&10\\ \hline\hline
	2&1&0&0&0&0&0&0&0&0\\ \hline
	3&2&3&0&0&0&0&0&0&0\\  \hline
	4&6&18&12&0&0&0&0&0&0\\  \hline
	5&12&72&120&60&0&0&0&0&0\\  \hline
	6&28 &267&780&900&360&0&0&0&0\\ \hline
	7&56&885&4188&8400&7560&2520&0&0&0\\  \hline
	8&120&2880&20400&63000&95760&70560&20160&0&0\\ \hline 9&240&9000&93120&417000&952560&1164240&725760&181440&0\\ \hline 10&496&27915&409140&2551440&8217720&14817600&15120000&8164800&1814400\\\hline
	\end {tabular}
	
	\medskip
	\caption{Some values of $\varphi_k(P_n)$ (Sequence A309785)}
\end{table}

\vfill

\begin{table}[h]
	\centering\small
	\label{tab_for_phi_k_of_C_n}
	\begin {tabular}{|l||l|l|l|l|l|l|l|l|l|} \hline
	\backslashbox{$n$}{$k$}&2&3&4&5&6&7&8&9&10\\ \hline\hline
	3&0   &1 &0&0&0&0&0&0&0\\  \hline
	4&0   & 3  &  3 &0&0&0&0&0&0\\  \hline
	5&0   &12   &24   & 12  &0&0&0&0&0\\  \hline
	6&1   &     34 & 124   &   150  & 60     &0&0&0&0\\ \hline
	7&  2 &     111 &  588     & 1200  &   1080  & 360  &0&0&0\\  \hline
	8& 6  & 315    & 2484    & 7845   & 11970   &8820  & 2520   &0&0\\ \hline 
	9&  14 & 933  &   10240    &  46280   &   105840   &  129360   &  80640  & 20160  &0\\ \hline 
	10& 30  &   2622  & 40464    &   254664  & 821592    & 1481760  & 1512000   & 816480    &  181440 \\\hline
	\end {tabular}
	
	\medskip
	\caption{Some values of $\varphi_k(C_n)$ (Sequence A309651)}
\end{table}

\pagebreak
\section*{Appendix B}\label{App_B}

\begin{center}
\textbf{Tables of $\Psi_k$, $\psi_k$, $\Pi_k$, $\pi_k$, $\Xi_k$ and $\xi_k$ for small paths and cycles}
\end{center}

\begin{table}[h]
	\centering\small
	\label{tab_for_big_Psi_k_of_P_n}
	\begin {tabular}{|l||l|l|l|l|l|l|l|l|l|} \hline
	\backslashbox{$n$}{$k$}&2&3&4&5&6&7&8&9&10\\ \hline\hline
	2&1&1&1&1&1&1&1&1&1\\ \hline
	3&1&2&2&2&2&2&2&2&2\\  \hline
	4&4& 8 &9&9&9&9&9&9&9\\  \hline
	5&6&20&26&27&27&27&27&27&27\\  \hline
	6&16 &     65 & 102 & 111 &112&112&112&112&112\\ \hline
	7& 28   & 182   & 364   & 440 & 452   &453&453&453&453\\  \hline
	8& 64  & 560  & 1436   & 1978   &  2120   & 2136  &2137&2137&2137\\ \hline 
	9&  120 &  1640   &    5560 &  9082  & 10428    &   10670 &    10690 &10691&10691\\ \hline 
	10& 256  & 4961    & 22136    &   43528   &  55039   &   58019    &   58409   & 58434   &58435\\\hline
	\end {tabular}
	
	\medskip
	\caption{Some values of $\Psi_k(P_n)$ (Sequence A309635)}
\end{table}

\vfill

\begin{table}[h]
	\centering\small
	\label{tab_for_big_Psi_k_of_C_n}
	\begin {tabular}{|l||l|l|l|l|l|l|l|l|l|} \hline
	\backslashbox{$n$}{$k$}&2&3&4&5&6&7&8&9&10\\ \hline\hline
	3&0&1&1&1&1&1&1&1&1\\  \hline
	4&0& 1 &2&2&2&2&2&2&2\\  \hline
	5&0&4&6&7&7&7&7&7&7\\  \hline
	6&1 &   9 & 19 & 22 &23&23&23&23&23\\ \hline
	7&  1   &  26    &  58  & 74   &  77  &78&78&78&78\\  \hline
	8& 4 &  66 &  195  &    279  &  306 &310&311&311&311\\ \hline 
	9& 7  &  183   &  651    & 1084   & 1255    &   1292  &  1296   &1297&1297\\ \hline 
	10& 18 &  488   & 2294    &     4554 &  5803   &   6141   & 6195     & 6200    &6201 \\\hline
	\end {tabular}
	
	\medskip
	\caption{Some values of $\Psi_k(C_n)$ (Sequence A309785)}
\end{table}

\begin{table}[h]
	\centering\small
	\label{tab_for_psi_k_of_P_n}
	\begin {tabular}{|l||l|l|l|l|l|l|l|l|l|} \hline
	\backslashbox{$n$}{$k$}&2&3&4&5&6&7&8&9&10\\ \hline\hline
	2&1&0&0&0&0&0&0&0&0\\ \hline
	3&1&1&0&0&0&0&0&0&0\\  \hline
	4&4& 4 &1&0&0&0&0&0&0\\  \hline
	5&6&14&6&1&0&0&0&0&0\\  \hline
	6&16 &     49 &  37 & 9 &1&0&0&0&0\\ \hline
	7& 28   & 154   & 182   & 76  & 12   &1&0&0&0\\  \hline
	8& 64  & 496  & 876   & 542   &  142   & 16  &1&0&0\\ \hline 
	9&  120 &  1520   &    3920 &  3522  & 1346    &   242 &    20 &1&0\\ \hline 
	10& 256  & 4705    & 17175    &   21392   &  11511   &  2980     &   390   & 25  &1\\\hline
	\end {tabular}
	
	\medskip
	\caption{Some values of $\psi_k(P_n)$ (Sequence A309748)}
\end{table}

\vfill

\begin{table}[h]
	\centering\small
	\label{tab_for_psi_k_of_C_n}
	\begin {tabular}{|l||l|l|l|l|l|l|l|l|l|} \hline
	\backslashbox{$n$}{$k$}&2&3&4&5&6&7&8&9&10\\ \hline\hline
	3&0&1&0&0&0&0&0&0&0\\  \hline
	4&0& 1 &1&0&0&0&0&0&0\\  \hline
	5&0&4&2&1&0&0&0&0&0\\  \hline
	6&1 &   8 & 10 & 3 &1&0&0&0&0\\ \hline
	7&  1   &  25    &  32   & 16   &  3  &1&0&0&0\\  \hline
	8& 4 &  62 &  129  & 84  &   27  &  4 &1&0&0\\ \hline 
	9& 7  &  176   &  468    & 433   & 171    &   37  &  4   &1&0\\ \hline 
	10& 18 &  470   & 1806    &     2260 &  1248   &  338     &   54   &  5  &1\\\hline
	\end {tabular}
	
	\medskip
	\caption{Some values of $\psi_k(C_n)$ (Sequence A309784)}
\end{table}

\newpage

\begin{table}[h]
	\centering\small
	\label{tab_for_Pi_k_of_P_n}
	\begin {tabular}{|l||l|l|l|l|l|l|l|l|l|l|} \hline
	\backslashbox{$n$}{$k$}&1&2&3&4&5&6&7&8&9&10\\ \hline\hline
	1&1&1&1&1&1&1&1&1&1&1\\ \hline
	2&1& 2  &2&2&2&2&2&2&2&2\\ \hline
	3&1& 3 &  4 &4&4&4&4&4&4&4\\  \hline
	4&1&6  &10    & 11  &11&11&11&11&11&11\\  \hline
	5&1&10  &  25 &  31 &  32 &32&32&32&32&32\\  \hline
	6&1& 20  &    70   &  107 &  116 & 117 &117&117&117&117\\ \hline
	7&1&   36  &   196  &   379  &455   & 467    & 468    &468&468&468\\  \hline
	8&1&  72  &  574  & 1451    & 1993    &  2135    &  2151  & 2152  & 2152 & 2152 \\ \hline 
	9&1& 136   & 1681     &  5611    &  9134   &  10480    & 10722   &   10742   & 10743  & 10743\\ \hline 
	10&1& 272   &   5002   &    22187  &  43580     &   55091  &   58071    &   58461   & 58486   & 58487\\\hline
	\end {tabular}
	
	\medskip
	\caption{Some values of $\Pi_k(P_n)$ (Sequence A320750)}
\end{table}

\begin{table}[h]
	\centering\small
	\label{tab_for_Pi_k_of_C_n}
	\begin {tabular}{|l||l|l|l|l|l|l|l|l|l|l|} \hline
	\backslashbox{$n$}{$k$}& 1&2&3&4&5&6&7&8&9&10\\ \hline\hline
	3& 1& 2 &  3 &3&3&3&3&3&3&3\\  \hline
	4& 1& 4  & 6    & 7  &7&7&7&7&7&7\\  \hline
	5& 1&4  &  9 &  11 &  12 &12&12&12&12&12\\  \hline
	6& 1& 8  &  22  &  33 &  36 & 37 &37&37&37&37\\ \hline
	7& 1&   9  &   40  &   73  &89   & 92    & 93    &93&93&93\\  \hline
	8& 1&  18  &  100 & 237    & 322  & 349  &  353  & 354  &354 & 354 \\ \hline 
	9& 1& 23   & 225     &  703    &  1137   &  1308    & 1345   &  1349   & 1350  & 1350\\ \hline 
	10& 1& 44   &   582   &    2433 &  4704  & 5953  &  6291   & 6345  & 6350  &6351 \\\hline
	\end {tabular}
	
	\medskip
	\caption{Some values of $\Pi_k(C_n)$ (Sequence A320748)}
\end{table}

\begin{table}[h]
	\centering\small
	\label{tab_for_pi_k_of_P_n}
	\begin {tabular}{|l||l|l|l|l|l|l|l|l|l|l|} \hline
	\backslashbox{$n$}{$k$}& 1&2&3&4&5&6&7&8&9&10\\ \hline\hline
	1&1&0&0&0&0&0&0&0&0&0\\ \hline
	2& 1& 1  &0&0&0&0&0&0&0&0\\ \hline
	3& 1& 2 &  1 &0&0&0&0&0&0&0\\  \hline
	4& 1&5  &4    & 1  &0&0&0&0&0&0\\  \hline
	5& 1&9  &  15 &  6 &  1 &0&0&0&0&0\\  \hline
	6& 1& 19  &    50   &  37 &  9 & 1 &0&0&0&0\\ \hline
	7& 1&   35  &   160  &   183  &76   & 12    & 1    &0&0&0\\  \hline
	8& 1&  71  &  502  & 877    & 542    &  142    &  16  & 1  & 0 & 0 \\ \hline 
	9& 1& 135   & 1545     &  3930    &  3523   &  1346    & 242   &   20   & 1  & 0\\ \hline 
	10& 1& 271   &   4730   &    17185  &  21393     &   11511  &   2980    &   390   & 25   &1 \\\hline
	\end {tabular}
	
	\medskip
	\caption{Some values of $\pi_k(P_n)$ (Sequence A284949)}
\end{table}

\begin{table}[h]
	\centering\small
	\label{tab_for_pi_k_of_C_n}
	\begin {tabular}{|l||l|l|l|l|l|l|l|l|l|l|} \hline
	\backslashbox{$n$}{$k$}& 1&2&3&4&5&6&7&8&9&10\\ \hline\hline
	3& 1& 1 &  1 &0&0&0&0&0&0&0\\  \hline
	4& 1& 3  & 2    & 1  &0&0&0&0&0&0\\  \hline
	5& 1&3  &  5 &  2 &  1 &0&0&0&0&0\\  \hline
	6& 1& 7  &    14   &  11 &  3 & 1 &0&0&0&0\\ \hline
	7& 1&   8  &   31  &   33  &16   & 3    & 1    &0&0&0\\  \hline
	8& 1&  17  &  82  & 137    & 85    &  27    &  4  & 1  & 0 & 0 \\ \hline 
	9& 1& 22   & 202     &  478    &  434   &  171    & 37   &   4   & 1  & 0\\ \hline 
	10& 1& 43   &   538   &    1851  &  2271  &  1249  &  338   & 54  & 5  &1 \\\hline
	\end {tabular}
	
	\medskip
	\caption{Some values of $\pi_k(C_n)$ (Sequence A152176)}
\end{table}

\newpage

\begin{table}[h]
	\centering\small
	\label{tab_for_Xi_k_of_P_n}
	\begin {tabular}{|l||l|l|l|l|l|l|l|l|l|l|} \hline
	\backslashbox{$n$}{$k$}&1&2&3&4&5&6&7&8&9&10\\ \hline\hline
	1&1&1 & 1  &1 & 1&1 &1 &1& 1&1 \\  \hline
	2& 0&0  &0&0&0&0&0&0&0&0\\ \hline
	3& 0&1 & 1  &1 & 1&1 &1 &1& 1&1 \\  \hline
	4& 0&2 & 4   & 4  &4 & 4& 4&4 & 4& 4\\  \hline
	5& 0&6 &  16 &  20 & 20  & 20&20 &20 & 20&20 \\  \hline
	6& 0& 12 &   52    &  80 &  86 & 86 & 86& 86& 86& 86\\ \hline
	7& 0& 28   &  169   &  336   &  400 &  409   &  409  &409 & 409 &409 \\  \hline
	8& 0&56   & 520   &  1344   &  1852   &  1976    &  1988  &1988  &1988 & 1988\\ \hline 
	9& 0&120   &   1600   &  5440    & 8868    & 10168     & 10388    &   10404   &10404   &10404 \\ \hline 
	10&0& 240   &   4840   &   21760   &   42892    &   54208   &  57108     & 57468    & 57488   & 57488 \\\hline
	\end {tabular}
	
	\medskip
	\caption{Some values of $\Xi_k(P_n)$ (Sequence A320751)}
\end{table}

\begin{table}[h]
	\centering\small
	\label{tab_for_Xi_k_of_C_n}
	\begin {tabular}{|l||l|l|l|l|l|l|l|l|l|} \hline
	\backslashbox{$n$}{$k$}&2&3&4&5&6&7&8&9&10\\ \hline\hline
	3& 0 & 0  &0 & 0&0 &0 &0& 0&0 \\  \hline
	4& 0 & 0  &0 & 0&0 &0 &0& 0&0 \\  \hline
	5& 0 & 0  &0 & 0&0 &0 &0& 0&0  \\  \hline
	6&  0 &   4    &  6 & 6 & 6 & 6& 6& 6& 6\\ \hline
	7& 1   &  13   &  30 & 34& 34   &  34  &34 &34 &34 \\  \hline
	8& 2   &  45   &  127  &  176 &185   &  185  &185  &185 & 185\\ \hline 
	9&  7   &  144    & 532    & 871     & 996 &1011   &   1011   &1011   &1011 \\ \hline 
	10&    12   &   416   &   1988    &   3982   &  5026     & 5280  &5304  & 5304   &5304\\\hline
	\end {tabular}
	
	\medskip
	\caption{Some values of $\Xi_k(C_n)$ (Sequence A324803)}
\end{table}

\begin{table}[h]
	\centering\small
	\label{tab_for_xi_k_of_P_n}
	\begin {tabular}{|l||l|l|l|l|l|l|l|l|l|} \hline
	\backslashbox{$n$}{$k$}&2&3&4&5&6&7&8&9&10\\ \hline\hline
	2& 0  &0&0&0&0&0&0&0&0\\ \hline
	3& 1 & 0  &0 & 0&0 &0 &0& 0&0 \\  \hline
	4& 2 & 2   & 0  &0 & 0& 0&0 & 0& 0\\  \hline
	5& 6 &  10 &  4 & 0  & 0&0 &0 & 0&0 \\  \hline
	6&  12 &   40    &  28 &  6 & 0 & 0& 0& 0& 0\\ \hline
	7&  28   &  141   &  167   &  64 &  9   &  0  &0 &0 &0 \\  \hline
	8& 56   & 464   &  824   &  508   &  124    &  12  &0  &0 & 0\\ \hline 
	9& 120   &   1480   &  3840    & 3428    & 1300     & 220    &   16   &0   &0 \\ \hline 
	10& 240   &   4600   &   16920   &   21132    &   11316   &  2900     & 360    & 20 &0\\\hline
	\end {tabular}
	
	\medskip
	\caption{Some values of $\xi_k(P_n)$ (Sequence A320525)}
\end{table}

\begin{table}[h]
	\centering\small
	\label{tab_for_xi_k_of_C_n}
	\begin {tabular}{|l||l|l|l|l|l|l|l|l|l|} \hline
	\backslashbox{$n$}{$k$}&2&3&4&5&6&7&8&9&10\\ \hline\hline
	3& 0 & 0  &0 & 0&0 &0 &0& 0&0 \\  \hline
	4& 0 & 0   & 0  &0 & 0& 0&0 & 0& 0\\  \hline
	5& 0 &  0 &  0 & 0  & 0&0 &0 & 0&0 \\  \hline
	6&  0 &   4    &  2 &  0 & 0 & 0& 0& 0& 0\\ \hline
	7&  1   &  12   &  17  &  4 &  0   &  0  &0 &0 &0 \\  \hline
	8& 2   & 43   &  82   &  49   &  9    &  0  &0  &0 & 0\\ \hline 
	9& 7   &   137   &  388    & 339    & 125    & 15    &   0   &0   &0 \\ \hline 
	10& 12 & 404  &   1572   &   1994   &   1044    &  254  &   24  &  0  & 0\\\hline
	\end {tabular}
	
	\medskip
	\caption{Some values of $\xi_k(C_n)$ (Sequence A324802)}
\end{table}

\end{document}